\documentclass[english]{llncs}
\usepackage[T1]{fontenc}
\usepackage[latin9]{inputenc}
\usepackage{babel}
\usepackage{verbatim}
\usepackage{mathrsfs}
\usepackage{enumitem}
\usepackage{amsmath}
\usepackage{amssymb}
\usepackage{stmaryrd}
\usepackage[all]{xy}
\usepackage[unicode=true,pdfusetitle,
 bookmarks=true,bookmarksnumbered=false,bookmarksopen=false,
 breaklinks=false,pdfborder={0 0 1},backref=false,colorlinks=false]
 {hyperref}

\makeatletter
\numberwithin{equation}{section}
\numberwithin{figure}{section}

\newdir{|>}{!/4.5pt/\dir{|}*:(1,-.2)\dir^{>}*:(1,+.2)\dir_{>}}

\makeatletter
\newcommand{\xyR}[1]{%
\makeatletter
\xydef@\xymatrixrowsep@{#1}
\makeatother
} 

\makeatletter
\newcommand{\xyC}[1]{%
\makeatletter
\xydef@\xymatrixcolsep@{#1}
\makeatother
} 

\makeatother

\begin{document}
\title{Saturated Kripke Structures as Vietoris Coalgebras}
\author{H. Peter Gumm, Mona Taheri}
\institute{Fachbereich Mathematik, Philipps-Universität Marburg}
\maketitle
\begin{abstract}
We show that the category of coalgebras for the compact Vietoris endofunctor
$\mathbb{V}$ on the category $Top$ of topological spaces and continuous
mappings is isomorphic to the category of all modally saturated Kripke
structures. Extending a result of Bezhanishvili, Fontaine and Venema
\cite{BezhFonVen10}, we also show that Vietoris subcoalgebras as
well as bisimulations admit topological closure and that the category
of Vietoris coalgebras has a terminal object.
\end{abstract}

\section{Introduction}

The theory of coalgebras has provided Computer Science with a much
needed general framework for dealing with all sorts of state based
systems, with their structure theories and their logics. The varied
types of systems, be they deterministic or nondeterministic automata,
transition systems, probabilistic or weighted systems, neighborhood
systems or the like, are fixed by the choice of an appropriate endofunctor
$F$ on the category of sets. From there on, with hardly any further
assumptions, a mathematically pleasing structure theory and corresponding
modal logics can be developed, see e.g. \cite{Rutten2000},\cite{Gum99b},\cite{Ihr03}.

A particularly well behaved situation arises when choosing for $F$
the finite-powerset functor $\mathbb{P}_{\omega}(-)$, perhaps augmented
with a constant component $\mathbb{P}(\Phi)$ representing sets of
atomic formulas. Coalgebras for the functor $\mathbb{P}_{\omega}(-)\times\mathbb{P}(\Phi)$
are precisely all image finite Kripke structures. Their logic is the
standard modal logic based on the atomic formulae in $\Phi$, and
they possess a terminal coalgebra $T$, even though its description
is always of an ``indirect'' nature (see \cite{Bar93,Bar94,GS2002}). 

The well known Hennessy-Milner theorem\cite{HM}, relating bisimulations
and logical equivalence is a consequence of image finiteness and will
not continue to hold for arbitrary Kripke structures, i.e. for coalgebras
of type $\mathbb{P}(-)\times\mathbb{P}(\Phi)$, see \cite{Hollenberg}. 

The theory of modal logic knows of a class of Kripke structures, which
lies between image finite structures and arbitrary Kripke structures
and which continues to enjoy the Hennessy-Milner theorem. These structures
are called \emph{modally saturated}, or simply $m$-saturated\cite{GorankoOtto}.
Unfortunately, though, there seems to be no $Set$-functor $F$, somehow
located in between $\mathbb{P}_{\omega}(-)\times\mathbb{P}(\Phi)$
and $\mathbb{P}(-)\times\mathbb{P}(\Phi)$, whose coalgebras would
be just the $m$-saturated Kripke structures.

It is well known, that much of the theory of coalgebras can be generalized
by turning to other categories than $Set$, provided they are co-complete
and come with a reasonable factorization structure. Some of the examples
studied in the literature replace the base category $Set$ with the
category $Rel$ of sets and relations \cite{jacobs_2016}, with the
category $Pos$ of posets \cite{BalanKurz} or $Cpo$ of complete
partial orders, with the category \emph{Meas} of measurable spaces
\cite{Doberkat}\cite{Moss2006}, or the category $Stone$ of Stone
spaces. Relevant to this present work will be the works of Kupke,
Kurz and Venema\cite{KupkeKurzVenema} as well as Bezhanishvili, Fontaine
and Venema \cite{BezhFonVen10} regarding coalgebras for the Vietoris
functor on the category of Stone spaces, i.e. compact zero-dimensional
Hausdorff spaces with continuous mappings.

When extending the Vietoris functor from Stone spaces to arbitrary
topological spaces $\mathcal{X}$, two natural choices offer themselves
for the object map: either the collection of all closed subsets of
$\mathcal{X}$ or the collection of all compact subsets of $\mathcal{X}$,
both equipped with appropriate topologies. Each of these choices yields
a functor, generalizing the mentioned Vietoris functor on Stone spaces.
Named the \emph{lower Vietoris functor}, resp. the \emph{compact Vietoris
functor}, these endofunctors on the category $Top$ of topological
spaces with continuous functions were explored in recent work by Hofmann,
Neves and Nora\cite{HofmannNevesNora}. 

For our investigation of saturated Kripke structures, the \emph{compact
Vietoris functor}, which we denote by $\mathbb{V}(-),$ turns out
to be appropriate. To model saturated Kripke-Structures, we choose
the endofunctor $\mathbb{V}(-)\times\mathbb{P}(\Phi)$ on the category
$Top$ of topological spaces and continuous mappings, where the $\mathbb{P}(\Phi)$-part
is a constant component equipped with an appropriate topology, intuitively
representing a set of atomic propositions, as above. We show that
$\mathbb{V}(-)\times\mathbb{P}(\Phi)$ coalgebras precisely correspond
to $m$-saturated Kripke structures, in fact there is an isomorphism
of categories between the category of saturated Kripke structures
and the category of all topological coalgebras for the compact Vietoris
functor $\mathbb{V}(-)\times\mathbb{P}(\Phi)$. 

This correspondence also yields a direct description of the terminal
$\mathbb{V}(-)\times\mathbb{P}(\Phi)$ coalgebra, which seems to be
simpler and more natural than the terminal $\mathbb{P}_{\omega}(-)\times\mathbb{P}(\Phi)$
coalgebra mentioned above: it is simply the Vietoris coalgebra corresponding
to the canonical model of normal modal logic over $\Phi.$

For Stone coalgebras we know from \cite{BezhFonVen10}, that the topological
closure $\bar{R}$ of a bisimulation $R$ is itself a bisimulation,
again. We verify that their arguments carry over to the more general
case of arbitrary Vietoris coalgebras, and we show also that a corresponding
result holds true for subcoalgebras in place of bisimulations. For
this we need to prepare some topological tools which may be interesting
in their own right, relating convergence in the Vietoris space $\mathbb{V}(\mathcal{X})$
to convergence in the base space $\mathcal{X}$. In particular, topological
nets $(\kappa_{i})_{i\in I}$ converging to $\kappa$ in the Vietoris
space $\mathbb{V}(\mathcal{X})$ are shown to correspond, up to subnet
formation, to nets $(a_{i})_{i\in I}$ with $a_{i}\in\kappa_{i}$,
converging in the base space $\mathcal{X}$ to $a\in\kappa$, and
conversely. 

\section{Preliminaries}

For the remainder of this article, we shall fix a set $\Phi,$ the
elements of which shall be called \emph{propositional variables} or
\emph{atomic propositions}.

\subsection{Kripke structures }
\begin{definition}
A \emph{Kripke structure (}also called \emph{Kripke model)} \textup{$\mathcal{X}=(X,R,v)$}
consists of a set $X$ of states together with a relation $R\subseteq X\times X$,
and a map $v:X\to\mathbb{P}(\Phi),$ where $\mathbb{P}$ denotes the
powerset functor.
\end{definition}

In applications, $X$ will typically be a set of possible states of
a system, $R$ is called the\emph{ transition relation}, describing
the allowed transitions between states from $X,$ and $v$ is called
the \emph{valuation}, since $v(x)$ consists of all atomic propositions
true in state $x.$ Instead of $(x,y)\in R$ we write $x\rightarrowtriangle y$
(or $x\rightarrowtriangle_{R}y$, if necessary). The idea is that
$x\rightarrowtriangle y$ expresses that it is possible for the system
to move from state $x$ to state $y$. Instead of a relation, we can
alternatively consider $R$ as a map $R:X\to\mathbb{P}(X)$. This
justifies the notation 
\[
R(x):=\{y\in X\mid(x,y)\in R\},
\]
so $R(x)$ denotes the \emph{successors} of $x,$ i.e. all states
reachable from $x$ in one step. 
\begin{definition}
For subsets $U\subseteq X$ define $\langle R\rangle U:=\{x\in X\mid\exists y\in U.(x,y)\in R\}$
and $[R]U:=\{x\in X\mid\forall y\in X.(x,y)\in R\implies y\in U\}.$
\end{definition}

Thus $x\in\langle R\rangle U$ if from $x$ it is\emph{ possible}
to reach an element of $U$ in one step, and $x\in[R]U$ says that
starting from $x,$ each transition will \emph{necessarily} take us
to $U$. Obviously, $\langle R\rangle(X-U)=X-[R]U$, and $[R](X-U)=X-\langle R\rangle U$,
so $\langle R\rangle$ and $[R]$ are mutually expressible if complements
are available.

\subsection{Modal logic}

Starting with the elements of $\Phi$\emph{ }as\emph{ atomic formulae,
}we obtain \emph{modal formulae} by combining them with the standard
boolean connectors $\wedge,$ $\vee,\neg$ or prefixing with the unary
modal operator $\square.$ We also allow the usual shorthands $\bigvee_{i\in I_{0}}\phi_{i}$
and $\bigwedge_{i\in I_{0}}\phi_{i}$, whenever $I_{0}$ is a finite
indexing set and each $\phi_{i}$ is a formula. Let $\mathcal{L}_{\Phi}$
be the set of all modal formulae so definable. 

\emph{Validity} $x\Vdash\phi,$ is defined for $x\in X$ and $\phi\in\mathcal{L}_{\Phi}$
in the usual way (see \cite{blackburn_rijke_venema_2001}):

\begin{eqnarray*}
x\Vdash p & :\iff & p\in v(x)\text{, \,\,\,\,\,whenever \ensuremath{p\in\Phi}}\\
x\models\square\phi & :\iff & \forall y\in X.\,(x\rightarrowtriangle y\implies y\models\phi).
\end{eqnarray*}
 For the boolean connectives $\wedge,\vee,\neg$, validity is defined
as expected. We extend it to sets of formulas $\Sigma\subseteq\mathcal{L}_{\Phi}$,
by 
\[
x\Vdash\Sigma\,\,:\iff\,\,\,\forall\phi\in\Sigma.\,x\Vdash\phi.\,\,\,\,\,\,
\]
For any $x\in X$ we put $\llbracket x\rrbracket:=\{\phi\in\mathcal{L}_{\Phi}\mid x\Vdash\phi\}$
and, similarly, for any $\phi\in\mathcal{L}_{\Phi}$~ we set $\llbracket\phi\rrbracket:=\{x\in X\mid x\Vdash\phi\}.$
Two elements $x,y$ from (possibly different) Kripke structures are
called \emph{logically equivalent} (in symbols $x\approx y$), if
for each formula $\phi\in\mathcal{L}_{\Phi}$ we have $x\models\phi\iff y\models\phi.$
Restricted to a single Kripke structure, $\approx$ is the kernel
of the \emph{semantic map }$x\mapsto\llbracket x\rrbracket$, and
hence an equivalence relation. Similarly, two modal formulae $\phi,\psi$
are \emph{equivalent}, and we write $\phi\equiv\psi,$ if for each
element $x$ in any Kripke structure we have $x\Vdash\phi\iff x\Vdash\psi.$ 

Adding a further modality $\diamondsuit$ to our logical language
by defining $\diamondsuit\phi:=\neg\square\neg\phi$ provides more
than only a convenient abbreviation. The resulting equivalences $\neg\square\phi\equiv\diamondsuit\neg\phi$
and $\neg\diamondsuit\phi\equiv\square\neg\phi$ allow one to push
negations inside, just as the deMorgan laws permit to do so for $\vee$
and $\wedge$, so that each modal formula becomes equivalent to a
modal formula in \emph{negation normal form} (\emph{nnf}), where negations
may only occur only in front of an atomic formula. We state this here
for later reference: 
\begin{lemma}
Every modal formula is equivalent to a modal formula in negation normal
form (nnf).
\end{lemma}

\subsection{Bisimulations}
\begin{definition}
\label{def:bisimulation}A \emph{bisimulation} between two Kripke
structures $\mathcal{X}_{1}=(X_{1},R_{1},v_{1})$ and $\mathcal{X}_{2}=(X_{2},R_{2},v_{2})$
is a relation $B\subseteqq X_{1}\times X_{2}$ such that for each
$(x,y)\in B$:
\begin{enumerate}
\item $v_{1}(x)=v_{2}(y)$,
\item $\forall x'\in X_{1}.\,x\rightarrowtriangle_{R_{1}}x'\implies\exists y'\in X_{2}.\,y\rightarrowtriangle_{R_{2}}y'\wedge\,x'B\,y'$,
\item $\forall y'\in X_{2}.\,y\rightarrowtriangle_{R_{2}}y'\implies\exists x'\in X_{1}.\,x\rightarrowtriangle_{R_{1}}x'\wedge\,x'B\,y'.$
\end{enumerate}
\end{definition}

The empty relation $\emptyset\subseteq X_{1}\times X_{2}$ is clearly
a bisimulation, and the union of a family of bisimulations between
$\mathcal{X}_{1}$ and $\mathcal{X}_{2}$ is again a bisimulation,
hence there is a largest bisimulation between $\mathcal{X}_{1}$ and
$\mathcal{X}_{2}$, which we call $\sim_{\mathcal{X}_{1},\mathcal{X}_{2}}$
or simply $\sim$, when $\mathcal{X}_{1}$ and $\mathcal{X}_{2}$
are clear from the context.

If $B_{1}\subseteq X_{1}\times X_{2}$ is a bisimulation between $\mathcal{X}_{1}$
and $\mathcal{X}_{2}$, then the converse relation $B_{1}^{-1}\subseteq X_{2}\times X_{1}$
is a bisimulation between $\mathcal{X}_{2}$ and $\mathcal{X}_{1}$.
Given another bisimulation $B_{2}$ between Kripke structures $\mathcal{X}_{2}$
and $\mathcal{X}_{3}$ then the relational composition $B_{1}\circ B_{2}$
is a bisimulation between $\mathcal{X}_{1}$ and $\mathcal{X}_{3}$.

A bisimulation \emph{on} a Kripke structure $\mathcal{X}=(X,R,v)$
is a bisimulation between $\mathcal{X}$ and itself. The identity
$\Delta_{X}=\{(x,x)\mid x\in X\}$ is always a bisimulation on $\mathcal{X}$.
Consequently, the largest bisimulation \emph{on} $\mathcal{X}$ is
an equivalence relation, denoted by $\sim_{\mathcal{X}}$ or simply
$\sim$. We say that two points $x\in X_{1}$ and $y\in X_{2}$ are
\emph{bisimilar,} if there exists a bisimulation $B$ with $x\,B\,y$,
which is the same as saying $x\sim y$. It is well known and easy
to check by induction: 
\begin{lemma}
\label{lem:Bisimilar_implies_logical_equivalence}Bisimilar points
satisfy the same formulae $\phi\in\mathcal{L}_{\Phi}$. 
\end{lemma}

A converse to this lemma was shown by Hennessy and Milner for the
case of \emph{image finite} Kripke structures. Here, an element $x$
in a Kripke structure $\mathcal{X}$ is called \emph{image finite}
if it has only finitely many successors, i.e. $\{x'\mid x\rightarrowtriangle x'\}$
is finite. $\mathcal{X}$ is called image finite if each $x$ from
$\mathcal{X}$ is image finite. Thus Hennessy and Milner proved in
\cite{HM}: 
\begin{proposition}
If $x\in\mathcal{X}$ and $y\in\mathcal{Y}$ are image finite elements,
then $x\sim y$ iff $x\approx y.$
\end{proposition}

\subsection{Homomorphisms and congruences}
\begin{definition}
\label{def:Homomorphism}A homomorphism $\varphi:\mathcal{X}\to\mathcal{Y}$
between Kripke structures $\mathcal{X}=(X,R_{\mathcal{X}},v_{\mathcal{X}})$
and $\mathcal{Y}=(Y,R_{\mathcal{Y}},v_{\mathcal{Y}})$ is a map whose
graph 
\[
G(\varphi):=\{(x,\varphi(x))\mid x\in X\}
\]
 is a bisimulation.\footnote{In the literature on Modal Logic (see e.g. \cite{blackburn_rijke_venema_2001},\cite{GorankoOtto}),
homomorphisms are usually called ``bounded morphisms''.} 

We call $\mathcal{X}$ a \emph{homomorphic preimage} of $\mathcal{Y}$,
and if $\varphi$ is surjective (which we indicate by writing $\varphi:\mathcal{X}\twoheadrightarrow\mathcal{Y}$)
then we call $\mathcal{Y}$ a \emph{homomorphic image} of $\mathcal{X}$.
If $X\subseteq Y$and the inclusion map $\iota:\mathcal{X}\to\mathcal{Y}$
is a homomorphism, then $\mathcal{X}$ is called a \emph{Kripke substructure}
of $\mathcal{Y}$. 
\end{definition}

It is easy to check that a subset $X\subseteq Y$ with the restrictions
of $R_{\mathcal{Y}}$ and $v_{\mathcal{Y}}$ to $X$ is a substructure
of $\mathcal{Y}$ if only if $R_{\mathcal{X}}(x)\subseteq X$ for
each $x\in X.$ If $\varphi:\mathcal{X}\to\mathcal{Y}$ is a homomorphism,
then its kernel 
\[
\ker\varphi:=\{(x,x')\in X\mid\varphi(x)=\varphi(x')\}
\]
is called a\emph{ congruence relation}. This is clearly an equivalence
relation and a bisimulation as well, since we can express it as a
relation product of $G(\varphi)$, the graph of $\varphi$, with its
converse $G(\varphi)^{-1}$ as 
\[
\ker\varphi=G(\varphi)\circ G(\varphi)^{-1}.
\]

\section{Saturated structures}

The notion of \emph{saturation} goes back to a similar concept of
Fine in \cite{Fine}. The terminology $m$\emph{-saturation }(or \emph{modal
saturation}) was adopted by \cite{blackburn_rijke_venema_2001} and
\cite{GorankoOtto}: 
\begin{definition}
An element $x$ is \emph{$m$-saturated}, if for each set $\Sigma$
of formulas, such that each finite subset $\Sigma_{0}\subseteq\Sigma$
is satisfied at some successor $y_{0}$ of $x,$ there is a successor
$y$ of $x$ satisfying all formulas in $\Sigma.$ A Kripke structure
is called \emph{$m$-saturated}, if each of its elements is saturated.
\end{definition}

In the following we shall find it convenient to informally use infinitary
disjunctions $\bigvee_{i\in I}\phi_{i}$ \textendash{} not as as a
logical expressions but as shorthands. In particular we write 
\[
x\Vdash\square\bigvee_{i\in I}\phi_{i}
\]
as an abbreviation for
\[
\forall y.(x\rightarrowtriangle y\implies\exists i\in I.\,y\models\phi_{i}).
\]

With this shorthand, the above definition can be reformulated: 
\begin{lemma}
An element $x$ in a Kripke model $\mathcal{X}=(X,R,v)$ is \emph{$m$}-saturated,
if for each family $(\phi_{i})_{i\in I}$ such that $x\Vdash\square\bigvee_{i\in I}\phi_{i}$
there exists a finite subset $I_{0}\subseteq I$ with $x\Vdash\square\bigvee_{i\in I_{0}}\phi_{i}$. 
\end{lemma}

Image finite elements are clearly saturated, but they are not the
only ones. Below, we consider two examples of Kripke structures. In
both cases, we assume $v(x):=\emptyset$ for each $x$:
\begin{example}
On the set $S:=\{s\}\cup\{s_{i}\mid i\in\mathbb{N}\}$ consider the
relation $R=\{(s,s_{i})\mid i\in\mathbb{N}\}\cup\{(s_{i+1},s_{i})\mid i\in\mathbb{N}\}$.
Then for each $s_{i}$ we have $s_{i}\Vdash\square^{i+1}\bot$, but
$s_{i}\not\Vdash\square^{j}\bot$ for $j\le i$. Therefore $(S,R,v)$
is not saturated, since $s\Vdash\square\bigvee_{i\in\mathbb{N}}(\square^{i+1}\bot)$,
but for no finite $I_{0}\subseteq\mathbb{N}$ do we have $s\Vdash\square\bigvee_{i\in I_{0}}(\square^{i+1}\bot)$.
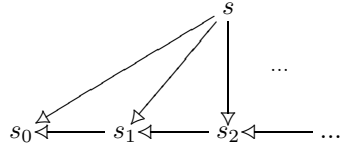
\[
\xyR{3pc}\xymatrix{ &  & s\ar@{-|>}[dll]\ar@{-|>}[dl]\ar@{-|>}[d]\ar@{}[dr]|{...}\\
s_{0} & s_{1}\ar@{-|>}[l] & s_{2}\ar@{-|>}[l] & ...\ar@{-|>}[l]
}
\]
\end{example}

Next, we modify the above structure by adding a ``point at infinity''
$s_{\infty}$ together with a self-loop $s_{\infty}\rightarrowtriangle s_{\infty}$
to obtain the following structure:
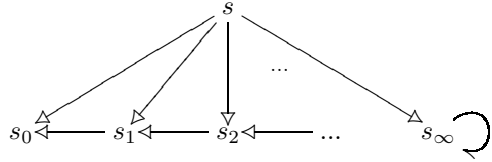
\begin{example}
\label{exa:compact_structure}
\[
\xymatrix{ &  & s\ar@{-|>}[dll]\ar@{-|>}[dl]\ar@{-|>}[d]\ar@{}[dr]|{...}\ar@{-|>}[drr]\\
s_{0} & s_{1}\ar@{-|>}[l] & s_{2}\ar@{-|>}[l] & ...\ar@{-|>}[l] & s_{\infty}\ar@(ur,dr)
}
\]
\end{example}

The point at infinity changes the situation. We claim:
\begin{lemma}
The Kripke structure in Example \ref{exa:compact_structure} is saturated.
\end{lemma}

\begin{proof}
We first observe that for $s_{\infty}$ and any formula $\phi$ we
have: 
\[
s_{\infty}\Vdash\diamondsuit\phi\iff s_{\infty}\Vdash\phi\iff s_{\infty}\Vdash\square\phi.
\]
Next we prove for each formula $\phi$:
\begin{claim}
\emph{If $s_{\infty}\Vdash\phi$, then there is some $k\in\mathbb{N}$
such that $s_{i}\Vdash\phi$ for each $i\ge k$.}
\end{claim}

We prove this claim by induction over the construction of nnf-formulae: 
\begin{itemize}
\item For $\phi=\bot$ and $\phi=\top,$ the claim is vacuously true. For
$\phi=\phi_{1}\wedge\phi_{2},$ from $s_{\infty}\Vdash\phi_{1}\wedge\phi_{2}$,
the hypothesis yields $k_{1}$ and $k_{2}$ such that $s_{i}\Vdash\phi_{1}$
for each $i\ge k_{1}$ and $s_{i}\Vdash\phi_{2}$ for each $i\ge k_{2}$.
With $k=max(k_{1},k_{2})$ we obtain $s_{i}\Vdash\phi_{1}\wedge\phi_{2}$
for $i\ge k$. For $\phi=\phi_{1}\vee\phi_{2}$ we could similarly
choose $k=min(k_{1},k_{2}).$
\item For $\phi=\square\phi_{1}$ we have $s_{\infty}\Vdash\phi\iff s_{\infty}\Vdash\phi_{1}$.
By assumption, there is some $k$ such that $s_{i}\Vdash\phi_{1}$
for each $i\ge k$. It follows that $s_{i}\Vdash\square\phi_{1}$
for $i\ge k+1.$ Similarly we argue for $\phi=\diamondsuit\phi_{1}$. 
\end{itemize}
Now, to show that $s$ in the structure of Example \ref{exa:compact_structure}
is saturated, assume that $s\Vdash\square\bigvee_{i\in I}\phi_{i}$,
then there is some $i_{\infty}\in I$ such that $s_{\infty}\Vdash\phi_{i_{\infty}}$.
The claim above provides a $k$ such that for each $j\ge k$ we have
$s_{j}\Vdash\phi_{i_{\infty}}$, and for each $j<k$ there is some
$i_{j}\in I$ with $s_{j}\Vdash\phi_{i_{j}}$. Altogether then with
$I_{0}:=\{i_{0},i_{1},...,i_{k-1}\}\cup\{i_{\infty}\}$ we have $s\Vdash\square\bigvee_{i\in I_{0}}\phi_{i}$.

Thus $s$ is saturated, and all other points in the structure are
image finite, hence they are saturated, too.
\end{proof}

We can extend Lemma \ref{lem:Bisimilar_implies_logical_equivalence}
to ``infinitary formulas'' in the following sense:
\begin{lemma}
\label{lem:Bisimulation_preserves_compactness}[Bisimulations preserve saturation]
If $B\subseteq X_{1}\times X_{2}$ is a bisimulation and $(x,y)\in B$,
then $x$ is saturated iff $y$ is saturated.
\end{lemma}

\begin{proof}
Assume that $x$ is saturated and $(x,y)\in B$. Suppose $y\Vdash\square\bigvee_{i\in I}\phi_{i}$,
then each $y'$ with $y\rightarrowtriangle y'$ satisfies one of the
formulas $\phi_{i}$. Each $x'$ with $x\rightarrowtriangle x'$ is
bisimilar to some $y'$ with $y\rightarrowtriangle y'$, so by Lemma
\ref{lem:Bisimilar_implies_logical_equivalence} each $x'$ satisfies
one of the $\phi_{i}$. This means that $x\Vdash\square\bigvee_{i\in I}\phi_{i}$.
By saturation of $x$ there is a finite subset $I_{0}\subseteq I$
with $x\Vdash\square\bigvee_{i\in I_{0}}\phi_{i}$. The latter, being
an honest modal formula, is preserved by bisimulation, so $y\Vdash\square\bigvee_{i\in I_{0}}\phi_{i}$.
\end{proof}

Lemma \ref{lem:Bisimilar_implies_logical_equivalence} implies that
for each $x\in\mathcal{X}$ and each formula $\phi$ we have 
\begin{equation}
x\Vdash\phi\iff\varphi(x)\Vdash\phi\label{eq:x approx phi(x)}
\end{equation}
 and Lemma \ref{lem:Bisimulation_preserves_compactness} tells us
that $x$ is saturated iff $\varphi(x)$ is saturated, which we might
combine to:
\begin{corollary}
\label{cor:Homos-preserve-saturation}Homomorphisms preserve and reflect
saturation.
\end{corollary}

On the level of Kripke structures, rather than elements, this translates
to:
\begin{corollary}
Homomorphic images and homomorphic preimages of saturated Kripke structures
are saturated.
\end{corollary}

Let $\mathcal{X}$ and $\mathcal{Y}$ be Kripke structures. Recall
that for elements $x\in X$ and $y\in Y$ we write $x\approx y$,
if they are logically equivalent, i.e. they satisfy the same modal
formulae. The following generalization of the Hennessy-Milner theorem
\cite{HM} is credited in \cite{blackburn_rijke_venema_2001} to unpublished
notes of Alfred Visser: 
\begin{proposition}
Let $\mathcal{X}$ and $\mathcal{Y}$ be saturated Kripke structures.
Then elements $x\in X$ and $y\in Y$ are bisimilar if and only if
they are logically equivalent. In short: $\sim_{\mathcal{X},\mathcal{Y}}\,\,\,=\,\,\,\approx_{\mathcal{X},\mathcal{Y}}$.
\end{proposition}

We shall next show that saturation allows us to describe the minimal
homomorphic image of a Kripke structure: 
\begin{lemma}
If $\mathcal{X}$ is saturated, then $\approx$ is a congruence relation
on $\mathcal{X}$.
\end{lemma}

\begin{proof}
Clearly, $\approx$ is an equivalence relation and therefore it is
the kernel of the map $\pi_{\approx}$ sending arbitrary elements
$x$ to $x/_{\approx},$ which denotes the equivalence class of $\approx$
containing $x$. To show that $\pi_{\approx}$ is a homomorphism,
we need to exhibit a coalgebra structure on $X/_{\approx}$, the factor
set of $X$. Put 
\begin{description}
\item [{$x/_{\approx}\,\Vdash\,p\,:\iff\,\exists\,\,x'\approx x\,.\,x'\Vdash p.$}]~
\item [{$x/_{\approx}\,\rightarrowtriangle\,y/_{\approx}$}] $\,:\iff\,$
there exist $x'\approx x$ and $y'\approx y$ such that $x'\rightarrowtriangle y'.$
\end{description}
We check that $\pi_{\approx}:\mathcal{X\to\mathcal{X}}/_{\approx}$
is indeed a Kripke homomorphism: 
\begin{itemize}
\item Clearly, $x\Vdash p$ iff $x/_{\approx}\Vdash p$ by definition of
$\Vdash$ on $X/_{\approx}$, and 
\item if $x\rightarrowtriangle y$, then $x/_{\approx}\,\rightarrowtriangle\,y/_{\approx}$
is also immediate by definition. Conversely, given $\pi_{\approx}(x)=x/_{\approx}\,\rightarrowtriangle\,y/_{\approx}$
for some $y,$ we must find a $y''$ with $x\rightarrowtriangle y''$
and $\pi_{\approx}(y'')=y/_{\approx}.$ Since $x/_{\approx}\,\rightarrowtriangle\,y/_{\approx},$
we know that there are $x'\approx x$ and $y'\approx y$ with $x'\rightarrowtriangle y'$.
By assumption, $\approx$~ is a bisimulation, so it follows that
there is some $y''$ with $x\rightarrowtriangle y''$ and $y''\approx y'$.
Consequently, $x\rightarrowtriangle y''$ and $\pi_{\approx}(y'')=\pi_{\approx}(y')=y/_{\approx},$
as required.
\end{itemize}
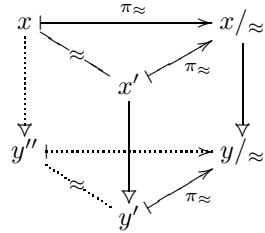
\[
\xyR{.8pc}\xymatrix{x\ar@{..|>}[dd]\ar@{-}[dr]|{\,\,\approx\,\,}\ar@{|->}[rr]^{\pi_{\approx}} &  & x/_{\approx}\ar@{-|>}[dd]\\
 & x'\ar@{-|>}[dd]\ar@{|->}[ur]_{\pi_{\approx}}\\
y''\ar@{..}[dr]|\approx\ar@{|..>}[rr] &  & y/_{\approx}\\
 & y'\ar@{|->}[ur]_{\pi_{\approx}}
}
\]
\\

Thus, $\pi_{\approx}$ is a homomorphism with kernel $\approx$, which
makes the latter a congruence relation.
\end{proof}

\begin{definition}
A Kripke structure is called \emph{simple}, if it does not have a
proper homomorphic image.
\end{definition}

Clearly, if $x\not\approx y$ then there cannot be a homomorphism
$\varphi$ with $\varphi(x)=\varphi(y)$, since $x\approx\varphi(x)$
and $y\approx\varphi(y)$. Thus, if $\approx$ is a congruence, $\mathcal{X}/_{\approx}$
must be simple. It follows:
\begin{theorem}
A Kripke structure is saturated iff it has a simple and saturated
homomorphic image. 
\end{theorem}

Observe that Example \ref{exa:compact_structure} is a Kripke structure,
which is saturated and simple, but not image finite. In particular
it does not have a homomorphism to an image finite Kripke structure.

\section{$F$-coalgebras}

Given a category $\mathscr{C}$ and an endofunctor $F:\mathscr{C}\to\mathscr{C},$
an $F$\emph{-coalgebra} $\mathcal{A}=(A,\alpha)$ is an object $A$
from $\mathscr{C}$ together with a morphism $\alpha:A\to F(A).$
The object $A$ is called the \emph{base object} and $\alpha$ is
called the \emph{structure morphism} of the $F$-coalgebra $\mathcal{A}=(A,\alpha).$

Given a second coalgebra $\mathcal{B}=(B,\beta)$, a \emph{homomorphism}
$\varphi:\mathcal{A}\to\mathcal{B}$ is a $\mathscr{C}$-morphism
$\varphi:A\to B$ which renders the following diagram commutative:
\[
\xyR{1.5pc}\xymatrix{A\ar[d]_{\alpha}\ar[r]^{\varphi} & B\ar[d]^{\beta}\\
F(A)\ar[r]_{F(\varphi)} & F(B)
}
\]

$F$-coalgebras with homomorphisms, as defined above, form a category,
which we shall call $\mathscr{C}_{F}$, or simply $Coalg_{F}$ when
the base category is understood. When $\varphi$ in the above figure
is a monomorphism in the base category, then we call $\mathcal{A}$
a \emph{subcoalgebra} of $\mathcal{B}$. 

Kripke structures are prime examples of coalgebras. Indeed, the successor
relation $R\subseteq X\times X$ can be understood as a map $R:X\to\mathbb{P}(X)$
and the valuation $v$ as a map $v:X\to\mathbb{P}(\Phi)$, where $\mathbb{P}$
is the powerset functor and $\Phi$ is the fixed set of propositional
atoms. Thus a Kripke structure $\mathcal{X}=(X,R,v)$ is simply an
$F$-coalgebra for the combined functor $\mathbb{P}(-)\times\mathbb{P}(\Phi)$,
that is a map 
\[
\alpha:X\to\mathbb{P}(X)\times\mathbb{P}(\Phi),
\]
whose first component models the successor relation $R$ and whose
second component is the valuation $v.$

It is easy to check (see \cite{Rut96}), that a homomorphism of Kripke
structures, as introduced earlier, is the same as a homomorphism of
coalgebras when Kripke structures are understood as $\mathbb{P}(-)\times\mathbb{P}(\Phi)$-coalgebras.

In this case a subcoalgebra $\mathcal{U}$ of $\mathcal{X}$ is uniquely
determined by its base set $U$. To be precise, $U\subseteq X$ carries
a subcoalgebra of the Kripke structure $\mathcal{X}=(X,\alpha)$ if
and only if $R(U)\subseteq U.$

Choosing the finite-powerset functor $\mathbb{P}_{\omega}(-)$ instead
of $\mathbb{P}(-)$, coalgebras for the functor $\mathbb{P}_{\omega}(-)\times\mathbb{P}(\Phi)$
are precisely the image finite Kripke structures.

Saturated Kripke structures, however, lying between image finite and
arbitrary Kripke structures, do not seem to allow such a simple modelling
by an appropriate $Set$-functor between $\mathbb{P}_{\omega(-)}$
and $\mathbb{P}(-).$ Instead, we shall have to pass to the category
\emph{Top} of topological spaces and continuous mappings and model
them as coalgebras over \emph{Top}.

\section{Topological models}
\begin{definition}
\label{def:TopologicalModel}A \emph{topological model} is a Kripke
model $\mathcal{X}=(X,R,v)$ together with a topology $\tau$ on $X$,
such that
\begin{enumerate}
\item $\forall x\in X.\,R(x)$ is compact
\item $\forall O\in\tau.\,\langle R\rangle O\in\tau$
\item $\forall O\in\tau.\,[R]O\in\tau$
\item $\forall p\in\Phi.\left\llbracket p\right\rrbracket \in\tau$ and
$(X-\left\llbracket p\right\rrbracket )\in\tau.$
\end{enumerate}
\end{definition}

A \emph{homomorphism} $\varphi:\mathcal{X}\to\mathcal{Y}$ between
topological models is simply a Kripke-homomorphism (see def. \ref{def:Homomorphism})
which additionally is continuous with respect to the topologies on
$\mathcal{X}$ and $\mathcal{Y}$.

We need two simple technical lemmas:
\begin{lemma}
\label{lem:closedSets}If $C$ is closed, then so are $\langle R\rangle C$
and $[R]C$.
\end{lemma}

\begin{proof}
Let $C=X-O$ where $O$ is open, then $\langle R\rangle C=\langle R\rangle(X-O)=X-[R]O$
and $[R]C=[R](X-O)=X-\langle R\rangle O.$
\end{proof}

\begin{lemma}
\label{lem:FormulasAreClopen}In every topological model the sets
$\left\llbracket \phi\right\rrbracket $ where $\phi\in\mathcal{L}_{\Phi}$,
are \emph{clopen} \emph{(clo}sed and o\emph{pen)}.
\end{lemma}

\begin{proof}
By induction on the construction of $\phi:$

For $p\in\Phi$ the assertion is part of the definition. If the claim
is true for $\phi,\phi_{1}$ and $\phi_{2}$, then it is obviously
true for all boolean compositions, in particular for $\neg\phi$ and
for $\phi_{1}\wedge\phi_{2}$.

Lemma \ref{lem:closedSets} and Definition \ref{def:TopologicalModel}
ensure that the claim remains true for $\square\phi$ and $\diamondsuit\phi,$
since $\left\llbracket \square\phi\right\rrbracket =[R]\left\llbracket \phi\right\rrbracket $
and for $\left\llbracket \diamondsuit\phi\right\rrbracket =\langle R\rangle\left\llbracket \phi\right\rrbracket .$
\end{proof}

Topological models with continuous Kripke-Homomorphisms obviously
form a category which we shall call $\mathscr{K}_{Top}$.

\section{The compact Vietoris-functor}

Leopold Vietoris, in his 1922 paper \cite{Viet1922}, defined his
\emph{domains of second order }(``Bereiche zweiter Ordnung'') as
the collection of closed subsets of a compact Hausdorff space. Later
several generalizations and modifications of this topology were introduced
and studied under the heading of \emph{hypertopolog}y.

In connection with Kripke structures, Bezhanishvili, Fontaine and
Venema \cite{BezhFonVen10} consider the Vietoris functor and Vietoris
coalgebras over Stone spaces, i.e. compact and totally disconnected
Hausdorff spaces. 

In compact Hausdorff spaces, all closed subsets are compact. Hence,
when extending the Vietoris functor to act on arbitrary topological
spaces $\mathcal{X}=(X,\tau)$, one has the choice to take as base
set for $\mathbb{V}(\mathcal{X})$ all closed subsets or all compact
subsets of $X.$ In \cite{HofmannNevesNora} the authors show that
both choices lead to endofunctors on the category $Top$ of topological
spaces, the \emph{``lower'' Vietoris functor}, and the \emph{compact
Vietoris functor}. Here we shall only need to work with the latter,
which for us then is ``the'' Vietoris functor: 

Given a topological space $\mathcal{X}=(X,\tau)$, the \emph{Vietoris
space} $\mathbb{V}(\mathcal{X})$ takes as base set the collection
of all compact subsets $K\subseteq X$. The \emph{Vietoris topology}
on $\mathbb{V}(\mathcal{X})$ is generated by a subbase consisting
of all sets
\begin{align*}
\langle O\rangle & :=\{K\in\mathbb{V}(\mathcal{X})\mid K\cap O\ne\emptyset\},\text{and}\\{}
[O] & :=\{K\in\mathbb{V}(\mathcal{X})\mid K\subseteq O\}
\end{align*}
where $O\in\tau$. 

If $\mathcal{X}=(X,\tau)$ and $\mathcal{Y}=(Y,\rho)$ are topological
spaces, then the \emph{Vietoris functor} sends a continuous function
$f:\mathcal{X}\to\mathcal{Y}$ to a map $(\mathbb{V}f):\mathbb{V}(\mathcal{X})\to\mathbb{V}(\mathcal{Y})$
by setting $(\mathbb{V}f)(K):=f(K).$ Recall that the image $f(K)$
of a compact set $K$ by a continuous map $f$ is always compact.
It is easy to calculate that $(\mathbb{V}f)^{-1}(\left\langle O\right\rangle )=\left\langle f^{-1}(O)\right\rangle $
and $(\mathbb{V}f)^{-1}(\left[O\right])=[f^{-1}(O)],$ hence $(\mathbb{V}f)$
is continuous with respect to the Vietoris topologies. In fact, $(\mathbb{V}f)^{-1}$
takes the defining subbase of $\mathbb{V}(\mathcal{Y})$ to the defining
subbase of $\mathbb{V}(\mathcal{X})$. This shows that $\mathbb{V}$
is indeed an endofunctor on $Top$.

Let now $\mathbb{P}(\Phi)$ be the powerset of $\Phi$, equipped with
the topology having as a base the set of all

\[
\uparrow\hspace{-2bp}p:=\{u\subseteq\Phi\mid p\in u\}
\]
where $p\in\Phi,$ together with their complements $\mathbb{P}(\Phi)\,-\uparrow\hspace{-2bp}p$.
This topology, trivially, is Hausdorff, but in general not compact.
\begin{definition}
The product $\mathbb{V}(-)\times\mathbb{P}(\Phi)$ of the Vietoris
functor $\mathbb{V}$ with the constant functor of value $\mathbb{P}(\Phi),$
carrying the above topology, will be called the $\Phi$-\emph{Vietoris
functor, or simply the }Vietoris functor\emph{, when $\Phi$ is clear. }
\end{definition}

The Vietoris functor is an endofunctor on the category $Top$ of topological
spaces with continuous maps. We can now define:

\emph{Vietoris coalgebras} are coalgebras over $Top$ for the $\Phi$-Vietoris
functor $\mathbb{V}(-)\times\mathbb{P}(\Phi)$, and the following
result shows that they agree with our topological models: 
\begin{theorem}
\label{thm:Vietoris=00003DTopMod}Vietoris coalgebras with coalgebra
homomorphisms are the same as topological models with continuous Kripke-homomorphisms.
\end{theorem}

\begin{proof}
Given a topological model $(X,R,v)$ with underlying space $\mathcal{X}=(X,\tau),$
we can consider it as a Vietoris coalgebra by defining the structure
map $\alpha:\mathcal{X}\to\mathbb{V}(\mathcal{X})\times\mathbb{P}(\Phi)$
as $\alpha(x):=(R(x),v(x)).$ To show that $\alpha$ is continuous,
we must verify that both components are continuous.

Continuity of (the map) $R:\mathcal{X}\to\mathbb{V}(\mathcal{X})$
needs to be tested only on the subbase for the Vietoris topology on
$\mathbb{V}(\mathcal{X})$. Indeed, assume $O\in\tau,$ then

\begin{align*}
R^{-1}(\left[O\right]) & =\{x\in X\mid R(x)\in\left[O\right]\}\\
 & =\{x\in X\mid R(x)\subseteq O\}=[R]O
\end{align*}
is open in $\tau$ and so is

\begin{align*}
R^{-1}(\left\langle O\right\rangle ) & =\{x\in X\mid R(x)\in\left\langle O\right\rangle \}\\
 & =\{x\in X\mid R(x)\cap O\ne\emptyset\}=\langle R\rangle O.
\end{align*}
To see that $v$ also is continuous, let $\uparrow\hspace{-2bp}p\,\subseteq\,\mathbb{P}(\Phi)$
be given, then

\[
v^{-1}(\uparrow\hspace{-2bp}p)=\{x\in X\mid p\in v(x)\}=\llbracket p\rrbracket\in\tau
\]
as well as 
\[
v^{-1}(\mathbb{P}(X)-\uparrow\hspace{-2bp}p)=\{x\in X\mid p\notin v(x)\}=(X-\llbracket p\rrbracket)\in\tau.
\]

Conversely, let $(X,\alpha)$ be a Vietoris coalgebra, with $\mathcal{X}=(X,\tau)$
as base space and $\alpha:\mathcal{X}\to\mathbb{V}(\mathcal{X})\times\mathbb{P}(\Phi)$
as structure morphism, then $\alpha=(R,v)$ with $R:=\pi_{1}\circ\alpha:\mathcal{X}\to\mathbb{V}(\mathcal{X})$
and $v:=\pi_{2}\circ\alpha:\mathcal{X}\to\mathbb{P}(\Phi)$, both
of which are continuous. Since $R(x)\in\mathbb{V}(\mathcal{X})$,
it is necessarily compact. If $O$ is open in $(X,\tau)$ then $\left[O\right]$
is open in $\mathbb{V}(\mathcal{X}),$ hence $R^{-1}([O])$ must be
open in $(X,\tau)$, hence so is

\[
[R]O=\{x\in X\mid R(x)\subseteq O\}=\{x\in X\mid R(x)\in[O]\}=R^{-1}([O]).
\]
Similarly, for $O$ open in $(X,\tau)$ we have $\left\langle O\right\rangle $
open in $\mathbb{V}(\mathcal{X}),$ hence $R^{-1}(\left\langle O\right\rangle )$
is open in $\mathcal{X},$ which means that

\[
\langle R\rangle O=\{x\in X\mid R(x)\cap O\ne\emptyset\}=\{x\in X\mid R(x)\in\left\langle O\right\rangle \}=R^{-1}(\left\langle O\right\rangle )
\]
is open as well.

Finally, for $p\in\Phi$ we have $\uparrow\hspace{-2bp}p=\{u\subseteq\Phi\mid p\in u\}$
clopen in the topology on $\mathbb{P}(\Phi)$, so also $\left\llbracket p\right\rrbracket =\{x\in X\mid p\in v(x)\}=v^{-1}(\{u\subseteq\Phi\mid p\in u\})=v^{-1}(\uparrow\hspace{-2bp}p)$
as well as its complement $X-\left\llbracket p\right\rrbracket $
are open in $\tau$.

Coalgebra homomorphisms between Vietoris coalgebras, as coalgebras
over $Top$, must be continuous and preserve both $R$ and $v$ which
means they are the same as continuous Kripke homomorphisms between
the corresponding topological models.
\end{proof}

\section{Characterization theorem}

The following theorem shows that saturated Kripke structures arise
precisely as the algebraic reducts of Vietoris coalgebras when forgetting
the topology. Bezhanishvili, Fontaine and Venema \cite{BezhFonVen10},
studying Vietoris coalgebras over Stone spaces, show that in this
case the underlying Kripke structures are saturated. In contrast to
their work, we consider the Vietoris functor over arbitrary topological
spaces, which allows us to obtain an equivalence:
\begin{theorem}
\label{thm:Compact Structures}For a Kripke structure $\mathcal{X}$
the following are equivalent:
\begin{enumerate}
\item $\mathcal{X}$ is saturated,
\item $\mathcal{X}$ is the algebraic reduct of a topological model,
\item $\mathcal{X}$ is the algebraic reduct of a Vietoris coalgebra.
\end{enumerate}
\end{theorem}

\begin{proof}
``$(1)\to(2)"$: Assuming that $\mathcal{X}=(X,R,v)$ is saturated,
let $\tau$ be the topology on $X$ generated by the sets $\left\llbracket \phi\right\rrbracket $
for $\phi\in\mathcal{L}_{\Phi}$. It follows that each $\left\llbracket \phi\right\rrbracket $
is clopen (\emph{clo}sed and \emph{open}), so each open set can be
written as $O=\bigcup_{i\in I}\left\llbracket \phi_{i}\right\rrbracket $
and each closed set as $C=\bigcap_{i\in I}\left\llbracket \phi_{i}\right\rrbracket $.

To show that $\mathcal{X}$ with this topology $\tau$ is a topological
model, we show first, that $R(x)$ is topologically compact. For that,
assume $R(x)\subseteq\bigcup_{i\in I}O_{i}$, then $R(x)\subseteq\bigcup_{i\in I}\bigcup_{j\in J_{i}}\left\llbracket \phi_{j}\right\rrbracket $,
i.e. 
\[
x\Vdash\square\bigvee_{i\in I}\bigvee_{j\in J_{i}}\phi_{j}.
\]
By saturation of $\mathcal{X}$, there are finitely many $j_{i_{1}}\in J_{i_{1}},...,j_{i_{n}}\in J_{i_{n}}$
with 
\[
x\Vdash\square(\phi_{j_{i_{1}}}\vee...\vee\phi_{j_{i_{n}}}),
\]
so $R(x)\subseteq O_{i_{1}}\cup...\cup O_{i_{n}}.$

Next, to see that $\langle R\rangle O$ is open, we calculate 
\begin{eqnarray*}
\langle R\rangle O & = & \langle R\rangle(\bigcup_{i\in I}\left\llbracket \phi_{i}\right\rrbracket )\\
 & = & \bigcup_{i\in I}\langle R\rangle\left\llbracket \phi_{i}\right\rrbracket \\
 & = & \bigcup_{i\in I}\{x\in X\mid x\models\diamondsuit\phi_{i}\}\phantom{{\text{\,for some very finite }I_{0}\subseteq I}}\\
 & = & \bigcup_{i\in I}\left\llbracket \diamondsuit\phi_{i}\right\rrbracket ,
\end{eqnarray*}
which is open, and similarly 

\begin{eqnarray*}
[R]O & = & [R](\bigcup_{i\in I}\left\llbracket \phi_{i}\right\rrbracket )\\
 & = & \{x\in X\mid R(x)\subseteq\bigcup_{i\in I}\left\llbracket \phi_{i}\right\rrbracket \}\\
 & = & \{x\in X\mid R(x)\subseteq\bigcup_{i\in J_{x}}\left\llbracket \phi_{i}\right\rrbracket \text{\,for some finite }J_{x}\subseteq I\}\\
 & = & \{x\in X\mid x\models\square\bigvee_{i\in J_{x}}\phi_{i}\,\text{for some finite}J_{x}\subseteq I\,\}\\
 & = & \bigcup_{J\subseteq I,\,J\text{\,finite}}\left\llbracket \square\bigvee_{i\in J}\phi_{i}\right\rrbracket ,
\end{eqnarray*}
which is open as well.

``$(2)\leftrightarrow(3)"$ is Theorem \ref{thm:Vietoris=00003DTopMod}.

``$(2)\to(1)"$ : Given a Kripke model $\mathcal{X}$ which is the
algebraic reduct of a topological model, assume $x\Vdash\square\bigvee_{i\in I}\phi_{i}$,
then $R(x)\subseteq\bigcup_{i\in I}\left\llbracket \phi_{i}\right\rrbracket $.
By Lemma \ref{lem:FormulasAreClopen}, the right hand side is a union
of open sets, thus by compactness of $R(x)$ there is a finite subset
$I_{0}\subseteq I$ with $R(x)\subseteq\bigcup_{i\in I_{0}}\phi_{i}$,
which means $x\Vdash\square\bigvee_{i\in I_{0}}\phi_{i}$.
\end{proof}

Given a saturated Kripke-structure $\mathcal{X}=(X,R,v)$, let $F(\mathcal{X})$
denote the Vietoris coalgebra, as constructed above, and conversely,
given a Vietoris coalgebra $\mathcal{A}$, let $G(\mathcal{A})$ be
the corresponding saturated Kripke structure. On objects, $F$ and
$G$ are clearly inverses to each other.

On morphisms, this is true as well, since a homomorphism $\varphi:\mathcal{X}\to\text{\ensuremath{\mathcal{Y}}}$
between saturated Kripke structures preserves (and reflects) modal
formulae (see Lemma \ref{lem:Bisimilar_implies_logical_equivalence})
and the topologies on $F(\mathcal{X})$ and $F(\mathcal{Y})$ are
generated by validity sets of formulae. Conversely, a morphism between
Vietoris coalgebras $\mathcal{A}$ and $\mathcal{B}$ is automatically
a Kripke-homomorphism by forgetting continuity.
\begin{corollary}
Saturated Kripke structures, topological models, and Vietoris coalgebras
are isomorphic as categories.
\end{corollary}

\section{Closure of Vietoris structures}

In those topological spaces where each point has a countable base
for its neighbourhoods, such as, for instance, in metric spaces, continuity
can be conveniently dealt with in terms of convergent sequences $(x_{n})_{n\in\mathbb{N}}.$
For general spaces $\mathcal{X}=(X,\tau)$, this intuitive approach
is not sufficient, but its spirit and its power can be salvaged if
one allows the linearly ordered set $\mathbb{\text{\ensuremath{\mathbb{N}}}}$,
indexing a sequence, to be replaced by arbitrary directed sets $I$
indexing the elements $(x_{i})_{i\in I}$ of a net. Often, a proof
based on convergence of sequences can be easily generalized by replacing
sequences with nets. Therefore net convergence can be considered more
intuitive than the equally powerful notion of filter convergence.
The following definitions and results on nets in general topological
spaces will be needed. They can be found as a series of exercises
in Munkres \cite{Munkres}.

\subsection{Nets and subnets}

A partially ordered set $(I,\le)$ is called \emph{directed}, if for
each pair $i_{1},i_{2}\in I$ there is some $i\in I$ such that $i_{1}\le i$
and $i_{2}\le i,$ i.e. $i$ is an upper bound for $\{i_{1},i_{2}\}.$
It follows that each finite subset $I_{0}\subseteq I$ has a common
upper bound.
\begin{definition}
A subset $J\subseteq I$ is called \emph{cofinal} in $I$, if for
each $i\in I$ there is some $j\in J$ with $i\le j.$ A map $f:J\to I$
between ordered sets $(J,\le)$ and $(I,\le)$ is called \emph{cofinal}
if its image $f[J]$ is cofinal in $I.$ 
\end{definition}

Clearly, if $J_{1}$ is cofinal in $J_{2}$ and $J_{2}$ cofinal in
$I$ then $J_{1}$ is cofinal in $I$. Also, compositions of cofinal
maps are cofinal.

Let $\mathcal{X}=(X,\tau)$ be a topological space and $x\in X$.
By $\mathfrak{U}(x)$ we denote the collection of all open neighborhoods
of $x.$ Observe that $\mathfrak{U}(x)$, when ordered by reverse
inclusion, is a directed set. 
\begin{definition}
A \emph{net} in $X$ is a map $\sigma:I\to X$ from a directed set
$I$ to the set $X.$ 
\end{definition}

If $\sigma(i)=x_{i}$, then one often denotes the net $\sigma$ as
$(x_{i})_{i\in I}$ and if $I$ is clear from the context one simply
writes $(x_{i}).$

The net $(x_{i})_{i\in I}$ \emph{converges} to $x\in X$ and we shall
write $\xymatrix{(x_{i})_{i\in I}\ar@{-^{>}}[r] & x,}
$ or when $I$ is understood, simply $\xymatrix{(x_{i}\ar@{-^{>}}[r] & x),}
$ provided that 
\begin{equation}
\forall U\in\mathfrak{U}(x).\,\exists i_{U}\in I.\,\forall i\ge i_{U}.\,x_{i}\in U.\label{eq:convergence-1}
\end{equation}
 In this case, $x$ is called a \emph{limit point} of $(x_{i})_{i\in I}.$
Colloquially, condition \ref{eq:convergence-1} can be expressed as
``$(x_{i})$\emph{ is eventually in every neighborhood of} $x".$

Limit points need not exist, nor need they be unique, unless $X$
is Hausdorff. In any case though, one has (see \cite{Munkres}):
\begin{proposition}
\label{continuity and closure with nets} Let $\mathcal{X}$ and $\mathcal{Y}$
be arbitrary topological spaces.
\end{proposition}

\begin{enumerate}
\item A map $\varphi:\mathcal{X}\to\mathcal{Y}$ is continuous at $x$ if
and only if it ``preserves convergence'', i.e. for all nets $(x_{i})_{i\in I}$
in $X:$
\[
\xymatrix{(x_{i}\ar@{-^{>}}[r] & x)}
\implies\xymatrix{(\varphi(x_{i})\ar@{-^{>}}[r] & \varphi(x))}
.
\]
\item Given a subset $A\subseteq X$, then $x\in X$ belongs to the topological
closure $\overline{A}$ of $A$ if and only if some net $(a_{i})$
in $A$ converges to $x.$ Thus, $A$ is closed iff it contains all
limit points of nets in $A.$ 
\end{enumerate}
$x\in X$ is called an \emph{accumulation point of the net $(x_{i})_{i\in I}$
}if\emph{ 
\begin{equation}
\forall U\in\mathfrak{U}(x).\,\forall i\in I.\,\exists j\ge i.\,x_{j}\in U.\label{eq:accumulation}
\end{equation}
}Condition (\ref{eq:accumulation}) can be phrased as: ``$x_{i}$
\emph{is frequently in every neighborhood of} $x".$%
{} A characterization of compactness using nets is (\cite{Munkres}):
\begin{lemma}
\emph{A subset $A\subseteq X$ is compact if and only if every net
in $A$ has an accumulation point in $A$.}
\end{lemma}

\begin{definition}
A net $\lambda:J\to X$ is a \emph{subnet} of $\sigma:I\to X$ if
there is a monotonic and cofinal map $f:J\to I$ with $\lambda=\sigma\circ f$:
\end{definition}

\[
\xymatrix{I\ar[r]^{\sigma} & X\\
J\ar@{-->}[u]^{f}\ar[ur]_{\lambda}
}
\]

Thus, if $\sigma=(x_{i})_{i\in I}$ then $\lambda=(x_{f(j)})_{j\in J}.$
One easily checks that the subnet relation is reflexive and transitive,
but mainly:
\begin{lemma}
\label{lem:subnet converges too=000023}If ($x_{i})_{i\in I}$ converges
to $x$ then so does each subnet $(x_{f(j)})_{j\in J}$.
\end{lemma}

\begin{lemma}
$x\in X$\emph{ is an accumulation point of the net }$\sigma:I\to X$
\emph{if and only if there is a subnet} $\lambda$ \emph{of }$\sigma$
\emph{converging to $x.$}
\end{lemma}

\begin{corollary}
\emph{A subset $A\subseteq X$ is compact iff every net in $A$ has
a subnet converging to some $a\in A.$}
\end{corollary}

\subsection{Convergence in Vietoris spaces}

In this section, we prepare our main result on net convergence in
Vietoris spaces. Let $\mathcal{X}=(X,\tau)$ be a topological space.
Recall that the \emph{Vietoris space }$\mathbb{V}(\mathcal{X})$ over
$\mathcal{X}$ consists of all compact subsets $K\subseteq X$, with
a topology generated by a subbase consisting of all sets

\vskip -0.4cm

\begin{eqnarray*}
\langle O\rangle & := & \{K\in\mathbb{V}(\mathcal{X})\mid K\cap O\ne\emptyset\}\\{}
[O] & := & \{K\in\mathbb{V}(\mathcal{X})\mid K\subseteq O]
\end{eqnarray*}
where $O$ ranges over all open subsets of $\mathcal{X}.$ The following
results establish the relevant connections between convergence in
$\mathbb{V}(\mathcal{X})$ and convergence in $\mathcal{X}=(X,\tau)$. 

\begin{lemma}
\label{lem:nonempty}Let $\kappa:I\to\mathbb{V}(\mathcal{X})$ be
a net in the Vietoris space. If $(\xymatrix{\kappa_{i}\ar@{-^{>}}[r] & K}
)$ and $K\ne\emptyset$, then $\kappa$ has a subnet, each member of
which is nonempty.
\end{lemma}

\begin{proof}
Since $K\ne\emptyset$, we have $K\in\langle X\rangle$, so $\langle X\rangle$
is a neighborhood of $K$ in $\mathbb{V}(\mathcal{X})$. As $\kappa$
converges to $K$, there must be some $i_{0}\in I$ such that $\forall i\ge i_{0}.\,\kappa_{i}\in\langle X\rangle,$
i.e $\forall i\ge i_{0}.\,\kappa_{i}\ne\emptyset.$ Put $J=(\{i\in I\mid i\ge i_{0}\},\le)$
and let $f:J\hookrightarrow I$ be the natural inclusion, then $f$
is clearly monotonic and cofinal. Therefore $\tau:=\kappa\circ f$
is a subnet of $\kappa$ and $\tau_{j}=\kappa_{f(j)}=\kappa_{j}\ne\emptyset,$
owing to $j\in J.$
\end{proof}

\begin{lemma}
\label{lem:K convergent forces elements convergent}Given a net $\kappa:I\to\mathbb{V}(\mathcal{X})$
converging to $K\in\mathbb{V}(\mathcal{X})$ and $b_{i}\in\kappa_{i}$
~for each $i\in I$. Then the net $(b_{i})_{i\in I}$ has a subnet
converging to some $b\in K$.
\end{lemma}

\begin{proof}
It is enough to show that $(b_{i})_{i\in I}$ has an accumulation
point $b\in K.$ For then we obtain a subnet $(b_{f(j)})_{j\in J}$
converging to $b$. By Lemma \ref{lem:subnet converges too=000023},
the subnet $(\kappa_{f(j)})_{j\in J}$ of $\kappa$ still converges
to $K.$

For every $x\in K$ which is not an accumulation point of $(b_{i})_{i\in I}$,
we obtain by negating (\ref{eq:accumulation}) an open neighborhood
$U_{x}$ of $x$ and an $i_{x}\in I$ such that for all $i\ge i_{x}$
we have $b_{i}\not\in U_{x}.$ Assuming that no $x\in K$ is an accumulation
point, then the family $(U_{x})_{x\in K}$ forms an open cover of
$K.$ By compactness, there is a finite subcover $U=U_{x_{1}}\cup...\cup U_{x_{n}}$.
Choose $i_{U}\ge i_{x_{1}},...,i_{x_{n}}$, then for every $i\ge i_{U}$
we have $b_{i}\not\in U\supseteq K.$

But $[U]$ is also an open neighborhood of $K$ in $\mathbb{V}(\mathcal{X})$
and $\xymatrix{(\kappa_{i}\ar@{-^{>}}[r] & K),}
$ so there exists $i_{[U]}$ with $\kappa_{i}\in[U]$, that is $b_{i}\in\kappa_{i}\subseteq U$
for $i\ge i_{[U]}$. For $i\ge\{i_{U},i_{[U]}\}$ we enter the contradiction
$b_{i}\in U$ and $b_{i}\not\in U$.
\end{proof}

\begin{lemma}
\label{lem:from K to a}Given a net $\kappa:I\to\mathbb{V}(\mathcal{X})$
converging to $K\in\mathbb{V}(\mathcal{X})$ and $a\in K$. Then there
is a subnet $(\kappa_{j})_{j\in J}$ and elements $a_{j}\in\kappa_{j}$
converging to $a.$
\end{lemma}

\begin{proof}
By Lemma \ref{lem:nonempty} and Lemma \ref{lem:subnet converges too=000023},
we may assume that $\kappa_{i}\ne\emptyset$ for all $i\in\mathcal{I}.$

For every $U\in\mathfrak{U}(a)$ we have $a\in U\cap K$, so $K\cap U\ne\emptyset,$
which means that $K\in\langle U\rangle,$ so $\langle U\rangle$ is
an open neighborhood of $K$ in $\mathbb{V}(\mathcal{X}).$

Since $\kappa$ converges to $K$ we have 
\begin{equation}
\forall U\in\mathfrak{U}(a).\exists i_{U}\in I.\forall i\ge i_{U}.\kappa_{i}\in\langle U\rangle.\label{eq:convergence}
\end{equation}
 Consider a partial order on 
\[
J:=\{(i,U)\in I\mathcal{\times\mathfrak{U}}(a)\mid\kappa_{i}\in\langle U\rangle\}
\]
 by defining: 
\[
(i_{1},U_{1})\le(i_{2},U_{2})\,:\iff i_{1}\le i_{2}\,\wedge\,U_{1}\supseteq U_{2}.
\]

To verify that $\mathcal{J}=(J,\le)$ is directed, let arbitrary $j_{1}=(i_{1},U_{1})$
and $j_{2}=(i_{2},U_{2})$ be given. Pick $U=U_{1}\cap U_{2}$ then
by (\ref{eq:convergence}) there is an $i_{U}\in I$ with $\kappa_{i}\in\langle U\rangle$
for all $i\ge i_{U}.$ It suffices to choose $i\ge i_{1},i_{2},i_{U},$
then $(i,U)\in J$ and $(i,U)\ge(i_{1},U_{1}),(i_{2},U_{2})$.

The map $\pi_{1}:J\to I$ given as $\pi_{1}(i,U):=i$ is clearly monotonic.
For each $i\in I$ we have $(i,X)\in J$ since $\kappa_{i}\ne\emptyset.$
Hence $\pi_{1}$ is cofinal. Therefore $\kappa\circ\pi_{1}:J\to\mathbb{V}(\mathcal{X})$
is a subnet of $\kappa$ and therefore also converges to $K$.

For each $(i,U)\in J$ we can pick some $a_{(i,U)}\in\kappa_{i}\cap U$.
This defines a net $(a_{j})_{j\in J}$ in $X$.

To show that $(a_{j})_{j\in J}$ converges to $a,$ let $U$ be any
open neighborhood of $a.$ By \ref{eq:convergence} there exist some
$i_{U}$ such that in particular $j_{U}:=(i_{U},U)\in J.$ We therefore
have $a_{j_{U}}:=a_{(i_{U},U)}\in U$ and for each $j=(i,U')\ge(i_{U},U)=j_{U},$
i.e. for $i\ge i_{U}$ and $U'\subseteq U$ we have $a_{j}=a_{(i,U')}\in\kappa_{i}\cap U'\subseteq U.$
\end{proof}

We can combine the previous two lemmas to a theorem relating convergence
in Vietoris spaces to convergence in their base spaces:
\begin{theorem}
\label{thm:Vietoris convergence}Let $(\kappa_{i})_{i\in I}$ converge
to $K$ in the Vietoris space $\mathbb{V}(\mathcal{X}).$ Then
\begin{enumerate}
\item for each $a\in K$ there is a subnet $(\kappa_{j})_{j\in J}$ and
elements $a_{j}\in\kappa_{j}$ such that $\xymatrix{(a_{j}\ar@{-^{>}}[r] & a),}
$ and
\item each net $(b_{i})_{i\in I}$ with $b_{i}\in\kappa_{i}$ has a subnet
$(b_{j})_{j\in J}$ converging to some $b\in K.$
\end{enumerate}
\end{theorem}

\subsection{Closure of subcoalgebras and bisimulations}

In this section we shall show that in topological Kripke structures,
i.e. for Vietoris coalgebras, the topological closure of a substructure
is again a substructure and the closure of a bisimulation is a bisimulation.
The second of these results has previously been shown for Vietoris
coalgebras over Stone spaces in \cite{BezhFonVen10}, but now we work
in the more general context of Vietoris coalgebras over arbitrary
topological spaces, so we were forced to prepare our tools in the
previous sections. 
\begin{theorem}
Let $\mathcal{A}=(A,\alpha)$ be a Vietoris coalgebra. If $U\subseteq A$
is a Kripke substructure of $\mathcal{A}$, then so is its topological
closure $\overline{U}$.
\end{theorem}

\begin{proof}
We may consider $\mathcal{A}$ as a topological model $(A,R,v_{A})$
where $R(a)=(\pi_{1}\circ\alpha)(a)$ for each $a\in A$ is the compact
set of all successors of $a.$ Assume that $U$ is a subcoalgebra,
i.e. a subset $U\subseteq X$ such that $R(u)\subseteq U$ for each
$u\in U.$ We need to show that the same property holds for $\overline{U}.$

Thus let $u\in\overline{U}$ be arbitrary and let $v$ be a successor
of $u$ in the coalgebra $\mathcal{X}$, i.e. $v\in R(u)$. We need
to show that $v\in\overline{U}.$

Due to Proposition \ref{continuity and closure with nets}, there
is a net $(u_{i})_{i\in I}$ converging to $u$ with each $u_{i}\in U.$
By continuity of $\alpha,$ the net $R(u_{i})_{i\in I}$ converges
to $R(u)$ in the Vietoris topology.

As $v\in R(u)$, we may assume by Lemma \ref{lem:nonempty}, that
each $R(u_{i})$ is nonempty. \footnote{With the phrase ``we may assume'' we often hide the technicality
that we might have to pass to a subnet, such as here to $(\alpha(u_{f(j)}))_{j\in J}$
and retroactively replace $(u_{i})_{i\in I}$ by the subnet $(u_{f(j)})_{j\in J}$,
which is always justified by Lemma \ref{lem:subnet converges too=000023}.} Next, we may assume by Theorem \ref{thm:Vietoris convergence} that
we can pick a $v_{i}$ from each $R(u_{i})$ so that the net $(v_{i})_{i\in I}$
converges to $v$ in $\mathcal{X}.$

Since $U$ was a subcoalgebra, $R(u_{i})\subseteq U,$ so each $v_{i}$
must belong to $U$. Therefore, we have found a net in $U$ which
converges to $v,$ hence $v\in\overline{U}.$
\end{proof}

\begin{theorem}
If $S$ is a Kripke bisimulation between Vietoris coalgebras $\mathcal{A}=(A,\alpha)$
and $\mathcal{B}=(B,\beta)$, then so is its topological closure $\overline{S}.$
\end{theorem}

\begin{proof}
Again, we consider $\mathcal{A}$ and $\mathcal{B}$ as topological
models with $\alpha=(R_{A},v_{A})$ and $\beta=(R_{B},v_{B})$. Given
$(a,b)\in\overline{S},$ we need to show that
\begin{enumerate}
\item $v_{A}(a)=v_{B}(b)$ and
\item whenever $a\rightarrowtriangle u$ then there is some $v$ with $b\rightarrowtriangle v$
and $(u,v)\in\overline{S}.$
\end{enumerate}
The third case of definition \ref{def:bisimulation} will follow by
a symmetric proof.

First note that by Theorem \ref{continuity and closure with nets}
there is a net $(a_{i},b_{i})_{i\in I}$ converging to $(a,b)$ with
each $(a_{i},b_{i})\in S.$ The individual nets $(a_{i})$, resp.
$(b_{i})$, converge to $a,$ resp. to $b$, since the projection
maps are continuous.

Also by continuity, $v_{A}(a_{i})$ and $v_{B}(b_{i})$ converge to
$v_{A}(a)$ and $v_{B}(b)\in\mathbb{P}(\Phi)$. Since $(a_{i},b_{i})\in S,$
we know $v_{A}(a_{i})=v_{B}(b_{i})$ for each $i\in I.$ Since the
topology on $\mathbb{P}(\Phi)$, the second component of the Vietoris
functor, is Hausdorff, we get $v_{A}(a)=v_{B}(b)$ as required.

Next, assume $a\rightarrowtriangle u$, i.e. $u\in R_{A}(a),$ then
we need to find some $v$ with $b\rightarrowtriangle v$ and $(u,v)\in\overline{S}.$

By continuity of $R_{A}$ and $R_{B}$, the nets $(R_{A}(a_{i}))_{i\in I}$
resp. $(R_{B}(b_{i})_{i\in I}),$ converge to $R_{A}(a),$ resp. to
$R_{B}(b)$ in the Vietoris spaces $\mathbb{V}(\mathcal{A})$, resp.
$\mathbb{V}(\mathcal{B})$.

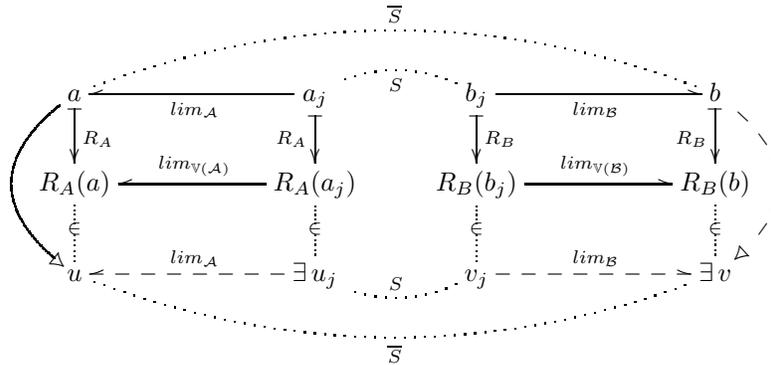
\[
\xymatrix{a\ar@/_{2pc}/@{-|>}[dd]\ar@{|->}[d]^{R_{A}}\ar@{..}@/^{2pc}/[rrrrr]^{\overline{S}} & \, & a_{j}\ar@{-_{>}}[ll]^{lim_{\mathcal{A}}}\ar@{|->}[d]_{R_{A}}\ar@/^{0.75pc}/@{..}[r]_{S} & b_{j}\ar@{-^{>}}[rr]_{lim_{\mathcal{B}}}\ar@{|->}[d]^{R_{B}} & \, & b\ar@/^{2pc}/@{--|>}[dd]\ar@{|->}[d]_{R_{B}}\\
R_{A}(a) & \, & R_{A}(a_{j})\ar@{-_{>}}[ll]_{lim_{\mathbb{V}(\mathcal{A})}} & R_{B}(b_{j})\ar@{-^{>}}[rr]^{lim_{\mathbb{V}(\mathcal{B})}} & \, & R_{B}(b)\\
u\ar@{..}[u]|\in\ar@{..}@/_{2pc}/[rrrrr]_{\overline{S}} &  & \exists\,u_{j}\ar@{--_{>}}[ll]_{lim_{\mathcal{A}}}\ar@{..}[u]|\in\ar@/_{0.75pc}/@{..}[r]^{S} & v_{j}\ar@{..}[u]|\in\ar@{--^{>}}[rr]^{lim_{\mathcal{B}}} &  & \exists\,v\ar@{..}[u]|\in
}
\]

In the sense mentioned previously, we may assume that the $R_{A}(a_{i})$
are nonempty and further, using Theorem \ref{thm:Vietoris convergence},
and possibly passing to a subnet indexed by some $J$, we find $u_{j}\in R_{A}(a_{j})$
with $\xymatrix{(u_{j}\ar@{-^{>}}[r] & u).}
$

Since $S$ is a bisimulation and $a_{j}\,S\,b_{j}$ for each $j$
and $u_{j}\in R_{A}(a_{j})$ it follows that there are $v_{j}\in R_{B}(b_{j})$
with $(u_{j},v_{j})\in S$ for each $j\in J.$ Since $\xymatrix{(b_{j}\ar@{-^{>}}[r] & b)}
$ it follows $\xymatrix{(R_{B}(b_{j})\ar@{-^{>}}[r] & R_{B}(b))}
$ by continuity of $R_{B}$. Therefore, Lemma \ref{lem:K convergent forces elements convergent}
forces $(v_{j})_{j\in J}$ to converge to some $v\in R_{B}(b).$

Consequently, $\xymatrix{((u_{j},v_{j})\ar@{-^{>}}[r] & (u,v))}
$ where $(u_{j},v_{j})\in S$ for each $j\in J,$ hence by Lemma \ref{continuity and closure with nets}
$(u,v)\in\overline{S}$ as desired.
\end{proof}

\section{The terminal Vietoris coalgebra}

To obtain the terminal Vietoris coalgebra, we utilize the equivalence
with saturated Kripke structures and look for a terminal saturated
Kripke structure instead. This will be found in the ``canonical model''.

Recall from \cite{blackburn_rijke_venema_2001}, that the \emph{canonical
model} for a normal modal logic consists of all \emph{maximally consistent}
subsets of $\mathcal{L}_{\Phi}$. Here $u\subseteq\mathcal{L}_{\Phi}$
is called maximally consistent, if
\begin{itemize}
\item no contradiction can be derived from the formulae in $u,$ and
\item for each formula $\phi\in\mathcal{L}_{\Phi}$, either $\phi\in u$
or $\neg\phi\in u.$
\end{itemize}
Typical sets of formulas which are maximally consistent arise as 
\[
\llbracket x\rrbracket:=\{\phi\mid x\Vdash\phi\},
\]
where $x$ is any element of any Kripke structure. Moreover, any consistent
set of formulas can be extended to a maximally consistent set.

It is also essential to know that a set $u$ is consistent, if and
only if every finite subset $u_{0}\subseteq u$ is consistent, see
\cite{blackburn_rijke_venema_2001}.

The \emph{canonical model} is now defined as $\mathcal{M}:=(M,\rightarrowtriangle_{\mathcal{M}},v_{\mathcal{M}})$
where $M$ is the collection of all maximally consistent subsets of
$\mathcal{L}_{\Phi}$, and $\rightarrowtriangle_{\mathcal{M}}$ and
$v_{\mathcal{M}}$ are defined as

\begin{equation}
u\rightarrowtriangle_{\mathcal{M}}v\,:\Leftrightarrow\,\forall\phi.(\square\phi\in u\implies\phi\in v),\label{eq:def:u --> v}
\end{equation}
 and 
\begin{equation}
v_{\mathcal{M}}(u)\,:=\,u\cap\Phi.\label{eq:def v in canonival Model}
\end{equation}

The latter definition extends to the important ``truth lemma'':
\begin{lemma}
For each formula $\phi\in\mathcal{L}_{\Phi}$ and each $u\in M$
we have: 
\[
u\Vdash\phi\iff\phi\in u.
\]
\end{lemma}

As an immediate corollary, we note:
\begin{corollary}
\label{cor:u approx v implies u=00003Dv-1}$\forall u,v\in M.\,u\approx v\implies u=v$.
\end{corollary}

First, we shall verify, that $\mathcal{M}$ is saturated: Given $u\in M$
and $\Sigma$ a set of formulas such that for every finite subset
$\Sigma_{0}\subseteq\Sigma$ there is some $v_{0}$ such that $u\rightarrowtriangle v_{0}$
and $v_{0}\Vdash\bigwedge\Sigma_{0}$. It follows that every finite
subset of the set 
\[
S\,:=\,\{\phi\mid\square\phi\in u\}\,\cup\,\Sigma
\]
is satisfied in some $v_{0}$, and hence consistent. Hence the whole
set $S$ itself is consistent. Let $v$ be any maximal consistent
set containing $S$, then $v\in M$ and clearly $u\rightarrowtriangle v$
as well as $v\Vdash\sigma$ for each $\sigma\in\Sigma.$ Therefore:
\begin{lemma}
\textup{$\mathcal{M}$ is saturated.}
\end{lemma}

Let us see that moreover:
\begin{theorem}
\textup{$\mathcal{M}$ is the terminal object in the category of all
saturated Kripke structures.}
\end{theorem}

\begin{proof}
First note that Corollary \ref{cor:u approx v implies u=00003Dv-1}
yields uniqueness: If for any Kripke structure $\mathcal{X}=(X,R,v)$
we had different homomorphisms $\varphi_{1},\varphi_{2}:\mathcal{X}\to\mathcal{M}$,
then for some $x\in X$ we would have $\varphi_{1}(x)\ne\varphi_{2}(x)$.
However, $x\approx\varphi_{1}(x)$ as well as $x\approx\varphi_{2}(x)$
according to \ref{eq:x approx phi(x)}, whence $\varphi_{1}(x)\approx\varphi_{2}(x)$,
which contradicts Corollary \ref{cor:u approx v implies u=00003Dv-1}.

For any Kripke structure $\mathcal{X}=(X,R,v)$ we show that the map
$\left\llbracket -\right\rrbracket :X\to M$ which sends an element
$x\in X$ to $\left\llbracket x\right\rrbracket :=\{\phi\mid x\Vdash\phi\}$
is a homomorphism, see definition \ref{def:Homomorphism}:

First, for each $p\in\Phi$ we have: $x\Vdash p$ in $\mathcal{X}$
implies $p\in\llbracket x\rrbracket$, so $\llbracket x\rrbracket\Vdash p$
in $\mathcal{M}$, by the \emph{Truth Lemma}.

Next, suppose $x,y\in\mathcal{X}$ and $x\rightarrowtriangle_{\mathcal{X}}y$.
Then for each $\phi\in\mathcal{L}_{\Phi}$ with $x\Vdash\square\phi$
it follows $y\Vdash\phi$, which by the truth lemma says $\square\phi\in\llbracket x\rrbracket\implies\phi\in\llbracket y\rrbracket$,
hence $\llbracket x\rrbracket\rightarrowtriangle_{\mathcal{M}}\llbracket y\rrbracket$
by \ref{eq:def:u --> v}.

Finally, let us assume $\llbracket x\rrbracket\rightarrowtriangle_{\mathcal{M}}v$
for some maximally consistent set $v.$ We need to find some $y\in\mathcal{X}$
with $x\rightarrowtriangle_{\mathcal{X}}y$ and $\llbracket y\rrbracket=v.$

\[
\xymatrix{x\ar@{|->}[rr]\ar@{-|>}[d] &  & \llbracket x\rrbracket\ar@{-|>}[d]\\
y_{i}\ar@{}[r]|\Vdash & \phi_{i}\ar@{}[r]|\notin & v
}
\]

For this we invoke a Hennessy\textendash Milner style argument again:
Let $(y_{i})_{i\in I}$ be the collection of all successors of $x.$
If $\llbracket y_{i}\rrbracket=v$ for some $i$, then we are done.
Otherwise, assume that $\llbracket y_{i}\rrbracket\ne v$ for each
$i\in I,$ then there are formulae $\phi_{i}$ with $\phi_{i}\in\llbracket y_{i}\rrbracket$
but $\phi_{i}\notin v$, or, in other words, $y_{i}\Vdash\phi_{i}$,
but $v\not\Vdash\phi_{i}$.

Hence $x\models\square\bigvee_{i\in I}\phi_{i}$. By assumption $\mathcal{X}$
is saturated, so $x\in X$ is saturated, which means that we can find
a finite subset $I_{0}\subseteq I$ with $x\Vdash\square\bigvee_{i\in I_{0}}\phi_{i}$.
This is now an honest formula, so from $\llbracket x\rrbracket\rightarrowtriangle v$,
and definition \ref{eq:def:u --> v} we conclude $v\models\bigvee_{i\in I_{0}}\phi_{i}$.
This means that $v\Vdash\phi_{i}$ for some $i\in I_{0}$, contradicting
our assumption.

Theorem \ref{thm:Compact Structures} tells us explicitly, how to
obtain the terminal Vietoris coalgebra, so we have:
\end{proof}

\begin{theorem}
The category of all Vietoris coalgebras has a terminal object. Its
base structure is the canonical model, consisting of all maximally
consistent sets of $\mathcal{L}_{\Phi}$-formulas, and its topology
is generated by the open sets $\{u\in\mathcal{M}\mid\phi\in u\}$
for all $\phi\in\mathcal{L}_{\Phi}.$
\end{theorem}

\section{Conclusion}

Starting from an arbitrary set $\Phi$ of atomic proposition, we have
characterized modally saturated Kripke structures as \emph{Top}-coalgebras
for $\mathbb{V}(-)\times\mathbb{P}(\Phi)$, which is the compact Vietoris
functor on the category Top of topological spaces and continuous mappings,
augmented with a constant part, representing sets of atomic propositions. 

In fact, the categories of saturated Kripke structures and the category
of all Vietoris coalgebras over the category $Top$ are isomorphic.
We have described the relation of convergence in the Vietoris space
$\mathbb{V}(\mathcal{X})$ to convergence in the base space $\mathcal{X},$
from which it was easy to derive that the Kripke-closure of bisimulations
and of subcoalgebras are again bisimulations, resp. subcoalgebras.
Finally, we have shown that the final Vietoris coalgebra exists, and
is derived from the canonical Kripke model.


\begin{thebibliography}{10}
\bibitem{BalanKurz} Adriana Balan and Alexander Kurz. \newblock Finitary functors: From set to preord and poset. \newblock In Andrea Corradini, Bartek Klin, and Corina C{\^i}rstea, editors,   {\em Algebra and Coalgebra in Computer Science}, pages 85--99, Berlin,   Heidelberg, 2011. Springer Berlin Heidelberg.
\bibitem{Bar93} M.~Barr. \newblock Terminal coalgebras in well-founded set theory. \newblock {\em Theoretical Computer Science}, (114(2)):299--315, 1993.
\bibitem{Bar94} M.~Barr. \newblock Additions and corrections to `{T}erminal coalgebras in well-founded   set theory'. \newblock {\em Theoretical Computer Science}, (124(1)):189--192, 1994.
\bibitem{BezhFonVen10} N.~Bezhanishvili, G.~Fontaine, and Y.~Venema. \newblock Vietoris bisimulations. \newblock {\em Journal of Logic and Computation}, 20(5), 2010.
\bibitem{blackburn_rijke_venema_2001} Patrick Blackburn, Maarten~de Rijke, and Yde Venema. \newblock {\em Modal Logic}. \newblock Cambridge Tracts in Theoretical Computer Science. Cambridge   University Press, 2001.
\bibitem{Doberkat} Ernst-Erich Doberkat. \newblock {\em Stochastic Coalgebraic Logic}. \newblock Monographs in Theoretical Computer Science. Springer Berlin   Heidelberg, 2009.
\bibitem{Fine} K.~Fine. \newblock Some connections between elementary and modal logic. \newblock In Patrick Blackburn, Johan Van~Benthem, and Frank Wolter, editors,   {\em Proceedings of the Third Scandinavian Logic Symposium}, volume~3, pages   15--31. North-Holland, 1973.
\bibitem{GorankoOtto} Valentin Goranko and Martin Otto. \newblock Model theory of modal logic. \newblock In Patrick Blackburn, Johan Van~Benthem, and Frank Wolter, editors,   {\em Handbook of Modal Logic}, volume~3, pages 249--329. Elsevier B.V., 2007.
\bibitem{Gum99b} H.~Peter Gumm. \newblock Elements of the general theory of coalgebras. \newblock In {\em LUATCS 99}. Rand Afrikaans University, Johannesburg, South   Africa, 1999.
\bibitem{Ihr03} H.~Peter Gumm. \newblock Anhang {\"u}ber universelle coalgebra. \newblock In Thomas Ihringer, editor, {\em Allgemeine Algebra}, volume~10 of   {\em Berliner Studienreihe zur Mathematik}. Heldermann Verlag, 2003.
\bibitem{GS2002} H.~Peter Gumm and Tobias Schr{\"o}der. \newblock Coalgebras of bounded type. \newblock {\em Mathematical Structures in Computer Science}, 12(5):565--578,   2002.
\bibitem{HM} Matthew Hennessy and Robin Milner. \newblock On observing nondeterminism and concurrency. \newblock In J.~de~Bakker and J.~van Leeuwen, editors, {\em Automata, Languages   and Programming. ICALP 1980}, volume~85 of {\em Lecture Notes in Computer   Science}, pages 159--173. Springer Verlag, Berlin, 1980.
\bibitem{HofmannNevesNora} Dirk Hofmann, Renato Neves, and Pedro Nora. \newblock Limits in categories of vietoris coalgebras. \newblock {\em Math. Struct. in Comp. Science}, pages 552--587, 2019.
\bibitem{Hollenberg} Marco Hollenberg. \newblock Hennessy-milner classes and process algebra. \newblock In A.~Ponse, de~Rijke~M., and Venema Y., editors, {\em Modal Logic   and Process Algebra}, volume~53 of {\em CSLI Lecture Notes}, pages 107--129.   CSLI Publications, 1995.
\bibitem{jacobs_2016} Bart Jacobs. \newblock {\em Introduction to Coalgebra: Towards Mathematics of States and   Observation}. \newblock Cambridge Tracts in Theoretical Computer Science. Cambridge   University Press, 2016.
\bibitem{KupkeKurzVenema} Clemens Kupke, Alexander Kurz, and Yde Venema. \newblock Stone coalgebras. \newblock {\em Theoretical Computer Science}, 327:109--134, 2004.
\bibitem{Moss2006} Lawrence~S. Moss and Ignacio~D. Viglizzo. \newblock Final coalgebras for functors on measurable spaces. \newblock {\em Information and Computation}, 204(4):610 -- 636, 2006. \newblock Seventh Workshop on Coalgebraic Methods in Computer Science 2004.
\bibitem{Munkres} James Munkres. \newblock {\em Topology}. \newblock Prentice Hall, 2000.
\bibitem{Rut96} J.J.M.M. Rutten. \newblock Universal coalgebra: a theory of systems. \newblock Technical report, CWI, Amsterdam, 1996.
\bibitem{Rutten2000} J.J.M.M. Rutten. \newblock Universal coalgebra: a theory of systems. \newblock {\em Theoretical Computer Science}, (249):3--80, 2000.
\bibitem{Viet1922} Leopold Vietoris. \newblock Bereiche zweiter {O}rdnung. \newblock {\em Monatsh. Math. Phys.}, 32(1), 1922.
\end{thebibliography}
\end{document}